\documentclass[11pt]{amsart}

\usepackage[cp1250]{inputenc}
\usepackage{graphicx}
\usepackage{url}
\usepackage[all]{xy}
\usepackage{multicol}
\usepackage{amsthm}
\usepackage{amsmath}
\usepackage{amssymb}
\usepackage{geometry}
\usepackage{pinlabel}
\usepackage{hyperref}

\newtheorem{Theorem}{Theorem}[section]

\newtheorem{definition}[Theorem]{Definition}

\newcommand{\ZZ}{\mathbb {Z}}
\newcommand{\RR}{\mathbb {R}}

\author{Eva Horvat}
\address{University of Ljubljana, Faculty of Education, Kardeljeva plo\v s\v cad 16, SI-1000 Ljubljana, Slovenia}
\email{eva.horvat@pef.uni-lj.si}

\begin{document}
\title{The label bracket $[ \cdot ]_{\mathcal{L}(\cdot )}$ for knotted trivalent graphs}

\maketitle

\begin{abstract}We generalize the construction from \cite{AM}, define the label bracket for knotted trivalent graphs in $\RR ^{3}$ and show it defines an isotopy invariant of such graphs.
\end{abstract}

\section{Introduction}

In \cite{AM}, the authors defined a combinatorial invariant of classical and virtual links which is called the \textit{label bracket}, and which generalizes several known link invariants. In this note, we generalize the idea of label bracket to obtain an invariant of knotted trivalent graphs. Our motivation is to provide a framework to study several open questions about the connection between the label bracket and the previously known quantum link invariants.\\ 

A \textit{knotted trivalent graph} is a regular trivalent graph embedded in $\RR ^{3}$. Two knotted trivalent graphs are called \textit{equivalent} (or \textit{ambient isotopic}) if there exists an isotopy of $\RR ^{3}$ that takes one onto the other. It is well-known that two knotted trivalent graphs are equivalent if and only if their diagrams are related by a finite sequence of the Reidemeister moves $\Omega _{1}-\Omega _{5}$, depicted in Fig. \ref{fig7} (for more details, we refer the reader to \cite{KA}). \\

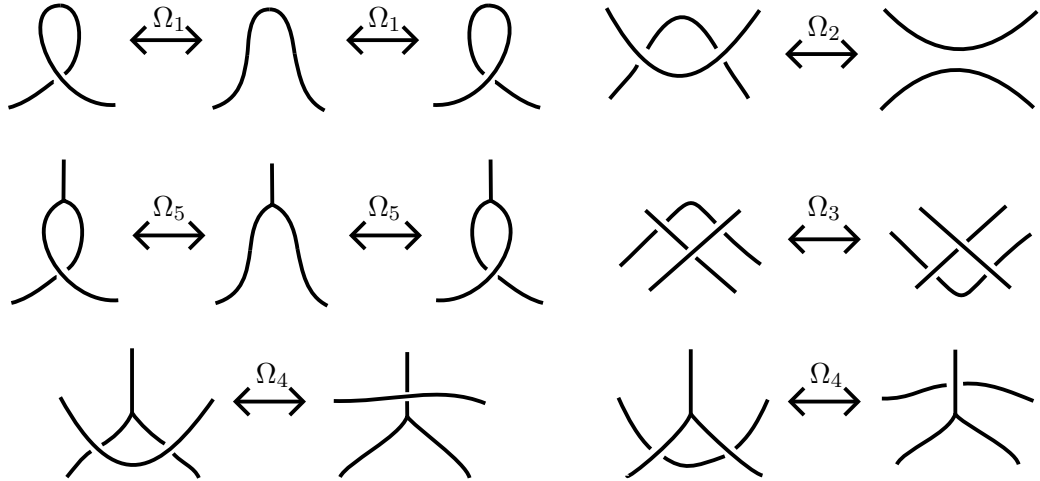
\begin{figure}[h!]
\labellist
\normalsize \hair 2pt
\pinlabel $\Omega _{1}$ at 360 1020
\pinlabel $\Omega _{1}$ at 840 1020
\pinlabel $\Omega _{2}$ at 1820 1000
\pinlabel $\Omega _{5}$ at 360 600
\pinlabel $\Omega _{5}$ at 840 600
\pinlabel $\Omega _{3}$ at 1820 600
\pinlabel $\Omega _{4}$ at 590 230
\pinlabel $\Omega _{4}$ at 1825 230
\endlabellist
\begin{center}
\includegraphics[scale=0.18]{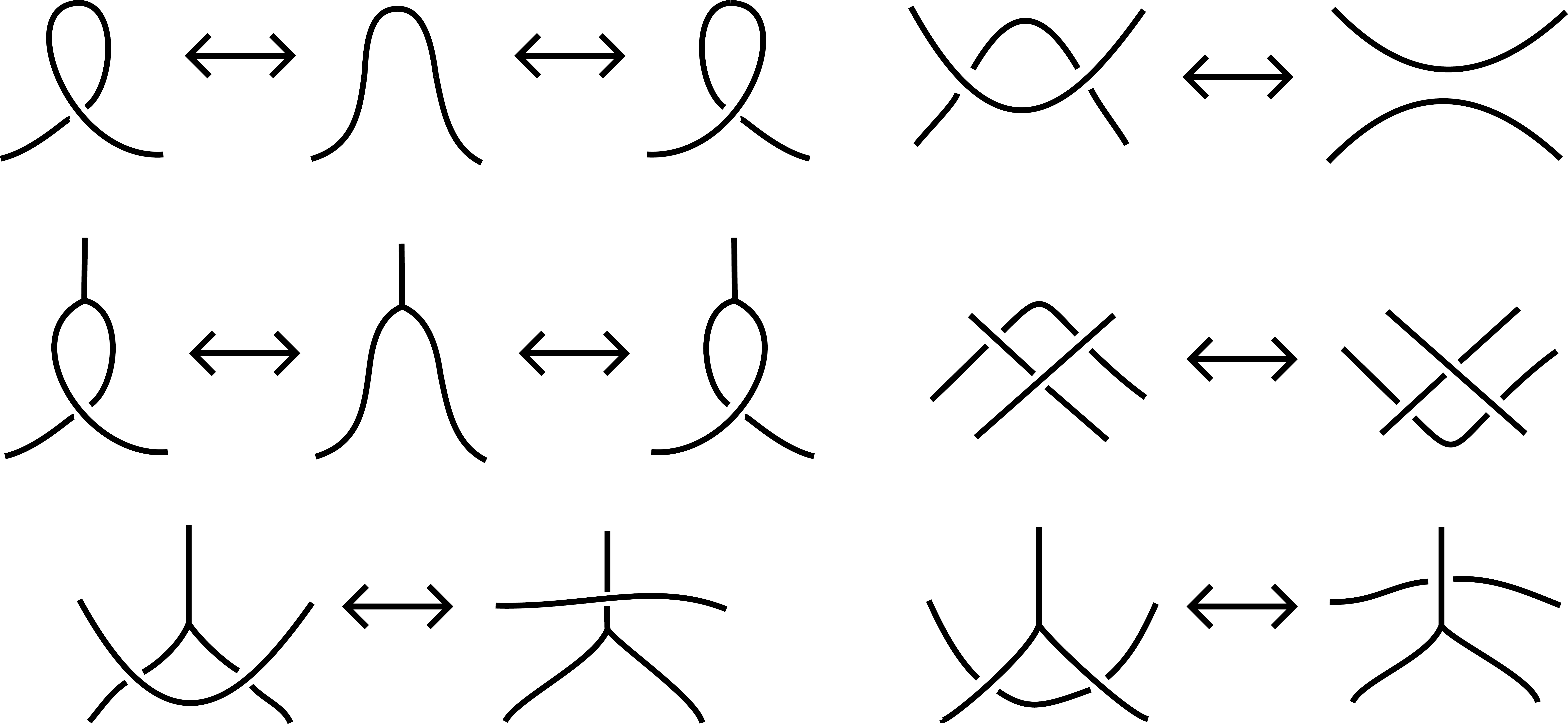}
\caption{Reidemeister moves $\Omega _{1}-\Omega _{5}$}
\label{fig7}
\end{center}
\end{figure}

The label bracket we define is an invariant of oriented trivalent graphs. We indicate orientation by adding a small arrow to every edge of the graph. We require that at any vertex, the arrows of the three incident edges either all point towards the vertex (the indegree equals 3), or all point away from the vertex (the indegree equals 0). 

\section{Generalization of the label bracket}

In this Section, we construct the label bracket for knotted trivalent graphs and show that it defines an isotopy invariant of such graphs. 

We generalize the notion of a \textit{classical label graph} from \cite{AM}. A \textit{classical label trigraph} will be a trivalent planar connected graph $G$, whose vertices are of types given in Fig. \ref{fig1}. 

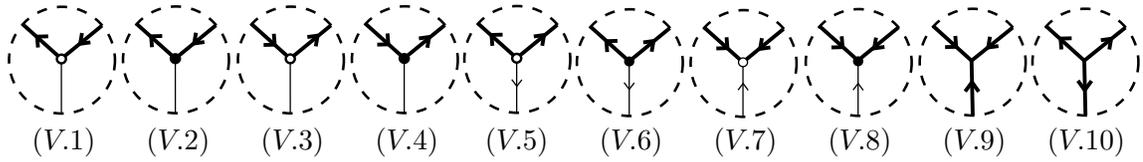
\begin{figure}[h!]
\labellist
\normalsize \hair 2pt
\pinlabel $(V.1)$ at 125 -20
\pinlabel $(V.2)$ at 380 -20
\pinlabel $(V.3)$ at 635 -20
\pinlabel $(V.4)$ at 890 -20
\pinlabel $(V.5)$ at 1135 -20
\pinlabel $(V.6)$ at 1390 -20
\pinlabel $(V.7)$ at 1645 -20
\pinlabel $(V.8)$ at 1900 -20
\pinlabel $(V.9)$ at 2155 -20
\pinlabel $(V.10)$ at 2405 -20
\endlabellist
\begin{center}
\includegraphics[scale=0.18]{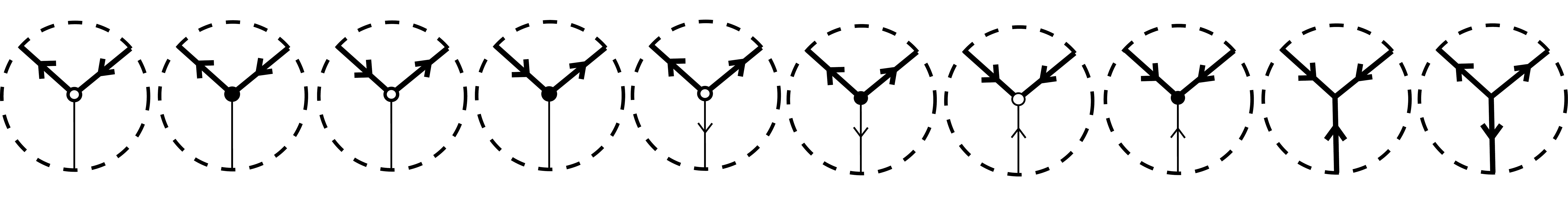}
\caption{Types of vertices of a classical label trigraph}
\label{fig1}
\end{center}
\end{figure}

Namely, each vertex is either unmarked (see types (V.9) and (V.10)) or it is marked by an empty or solid circle. An unmarked vertex is incident to three thick lines. A marked vertex is incident to one thin and two thick lines. All the edges are oriented except for thin edges that are adjacent to thick edges with coherent orientation. The indegree of the vertices of types (V.5), (V.6) and (V.10) equals 0, while the indegree of the vertices of types (V.7), (V.8) and (V.9) equals 3.   

\begin{figure}[h!]
\labellist
\normalsize \hair 2pt
\pinlabel $\mathcal{L}$ at 480 0
\pinlabel $\mathcal{L}$ at 1500 0
\pinlabel $\mathcal{L}$ at 480 490
\pinlabel $\mathcal{L}$ at 1500 490
\pinlabel $R_{S.1}\colon $ at -80 700
\pinlabel $R_{S.2}\colon $ at -80 210
\endlabellist
\begin{center}
\includegraphics[scale=0.20]{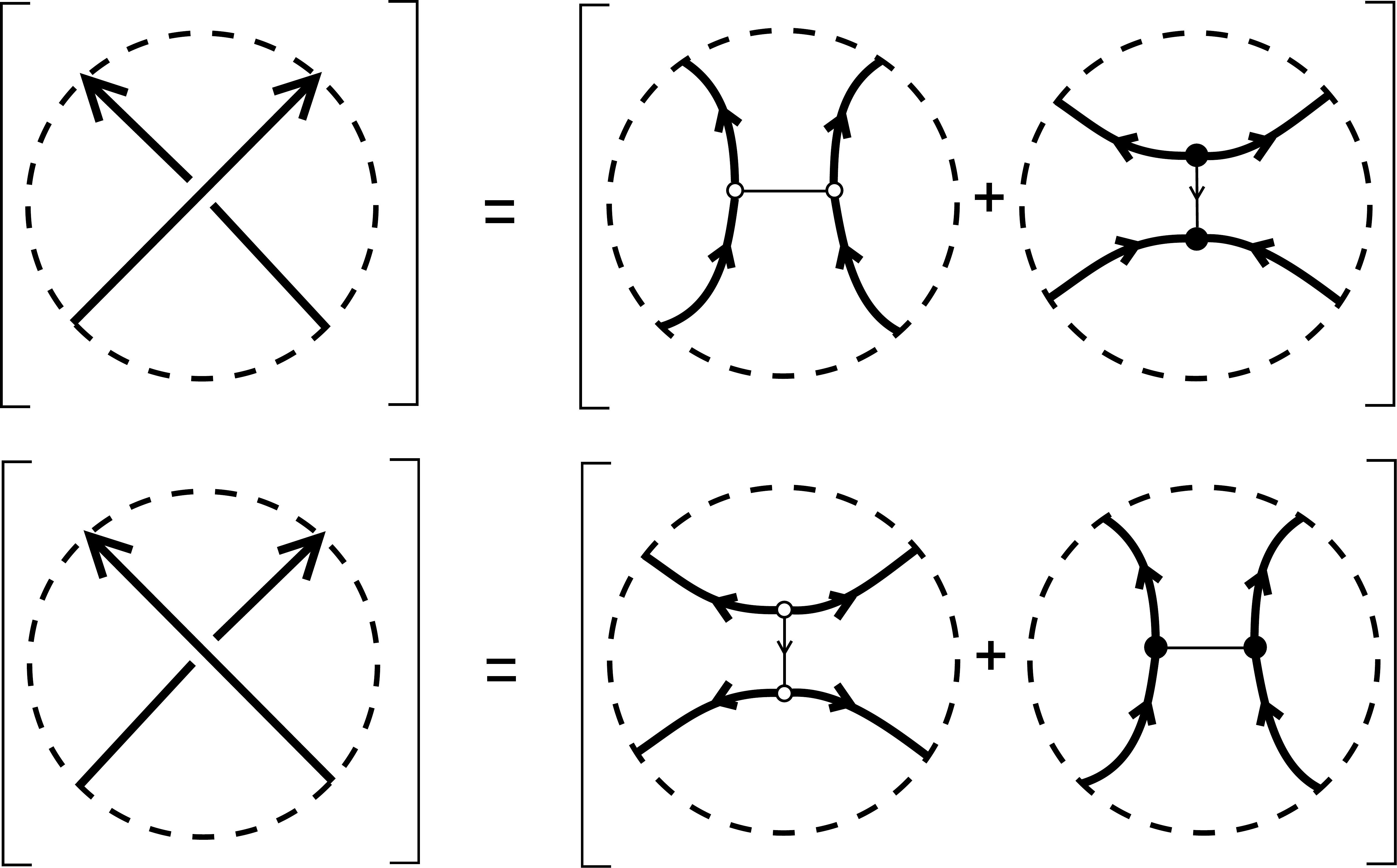}
\caption{Smoothing relations $R_{S.1}$ and $R_{S.2}$ of the label bracket, defined in \cite{AM}}
\label{fig6}
\end{center}
\end{figure}
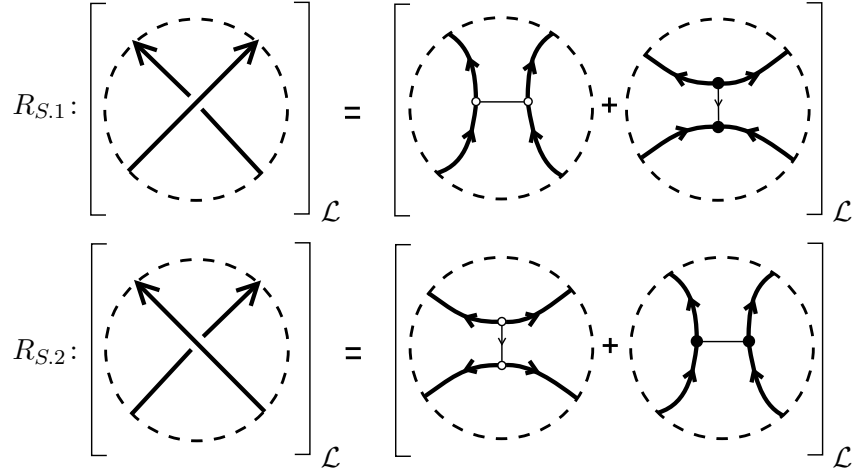

Denote by $C(G)$ a $\ZZ $-module, generated by the classical label trigraphs modulo relations 
\begin{xalignat}{1}\label{rel}
& R_{1.1}, R_{1.2}, R_{2.j}, R_{3.1}\textrm{ (given in \cite{AM}), $R_{4.j}$ (Fig. \ref{fig2}) and $R_{5.j}$ (Fig. \ref{fig4}) for $j=1,2,3,4$.}
\end{xalignat} 
Let $D$ be a diagram of a knotted trivalent graph in $\RR ^{3}$. If $D$ has $n$ crossings, the smoothing relations in Fig. \ref{fig6} define $2^{n}$ states $G_{s}(D)\in C(D)$ of the diagram $D$. 

\begin{definition} The \textit{(extended) label bracket} of an oriented diagram $D$ is given by $$[D]_{\mathcal{L}}=\sum _{s=1}^{2^{n}}G_{s}(D)$$ modulo relations \eqref{rel}.
\end{definition}

In \cite{AM}, the authors have proved that the label bracket is invariant under the Reidemeister moves $\Omega _{1}$, $\Omega _{2}$ and $\Omega _{3}$, and thus represents an invariant of classical links. The extended label bracket provides a generalization of this result. 

\begin{figure}[h!]
\labellist
\normalsize \hair 2pt
\pinlabel $\mathcal{L}$ at 1715 1440
\pinlabel $\mathcal{L}$ at 2675 1440
\pinlabel $\mathcal{L}$ at 1715 960
\pinlabel $\mathcal{L}$ at 2675 960
\pinlabel $\mathcal{L}$ at 1715 480
\pinlabel $\mathcal{L}$ at 2675 480
\pinlabel $\mathcal{L}$ at 1715 0
\pinlabel $\mathcal{L}$ at 2675 0
\pinlabel $R_{4.1}\colon $ at -80 1660
\pinlabel $R_{4.2}\colon $ at -80 1180
\pinlabel $R_{4.3}\colon $ at -80 700
\pinlabel $R_{4.4}\colon $ at -80 230
\endlabellist
\begin{center}
\includegraphics[scale=0.18]{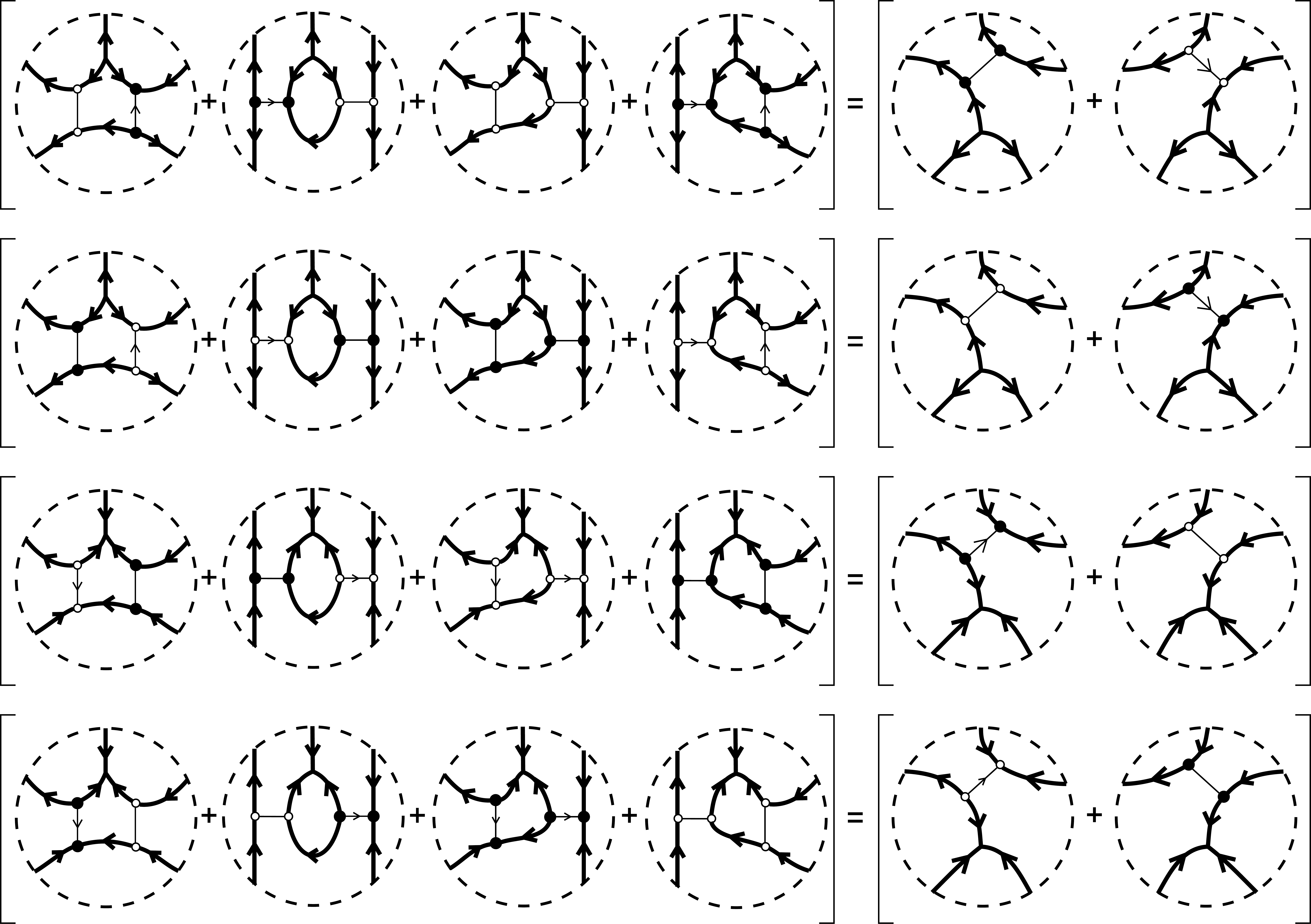}
\caption{Relations $R_{4.1}-R_{4.4}$ of the label bracket}
\label{fig2}
\end{center}
\end{figure}
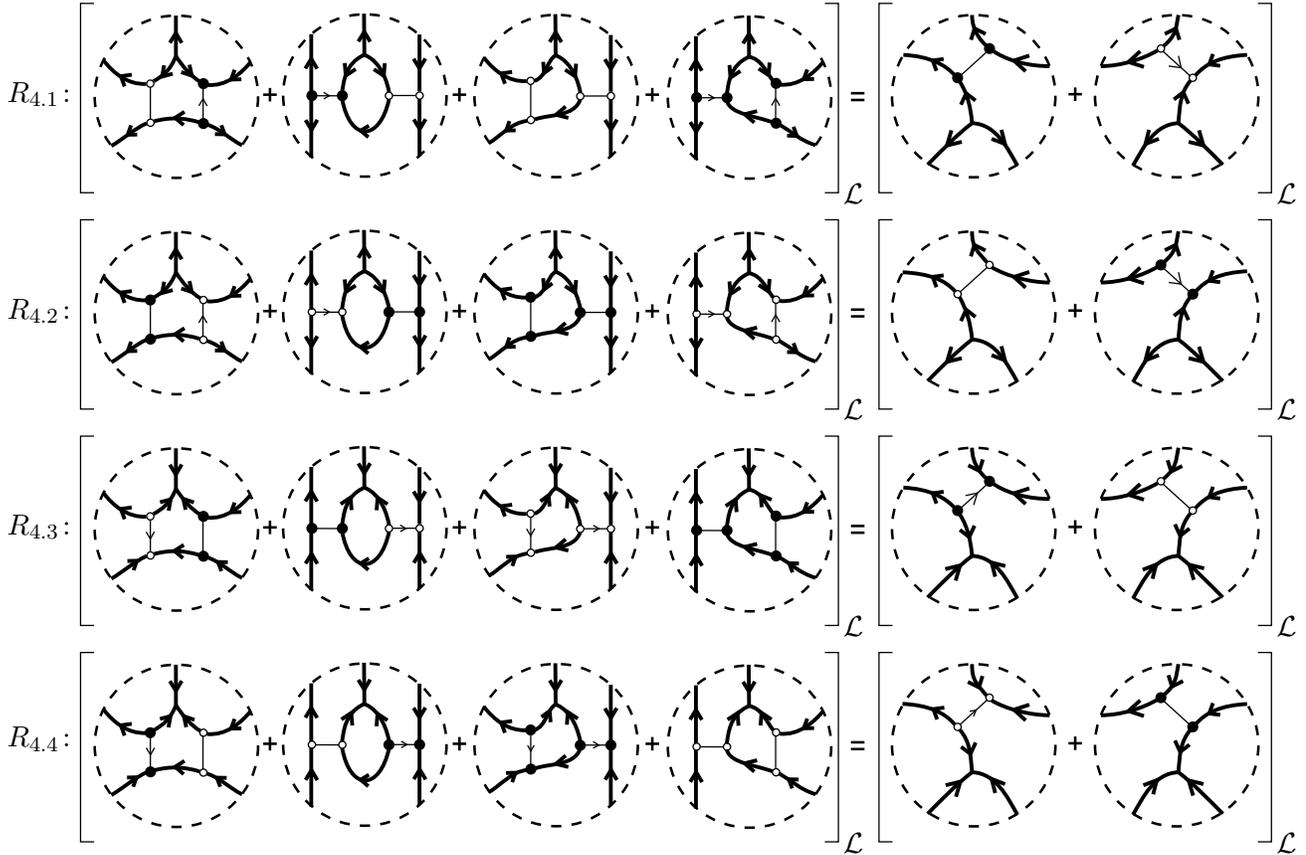

\begin{figure}[h!]
\labellist
\normalsize \hair 2pt
\pinlabel $\mathcal{L}$ at 880 0
\pinlabel $\mathcal{L}$ at 1490 0
\pinlabel $\mathcal{L}$ at 2510 0
\pinlabel $\mathcal{L}$ at 880 470
\pinlabel $\mathcal{L}$ at 1490 470
\pinlabel $\mathcal{L}$ at 2510 470
\pinlabel $R_{5.1}$ at 940 730
\pinlabel $R_{5.2}$ at 1540 730
\pinlabel $R_{5.3}$ at 940 280
\pinlabel $R_{5.4}$ at 1540 290
\endlabellist
\begin{center}
\includegraphics[scale=0.18]{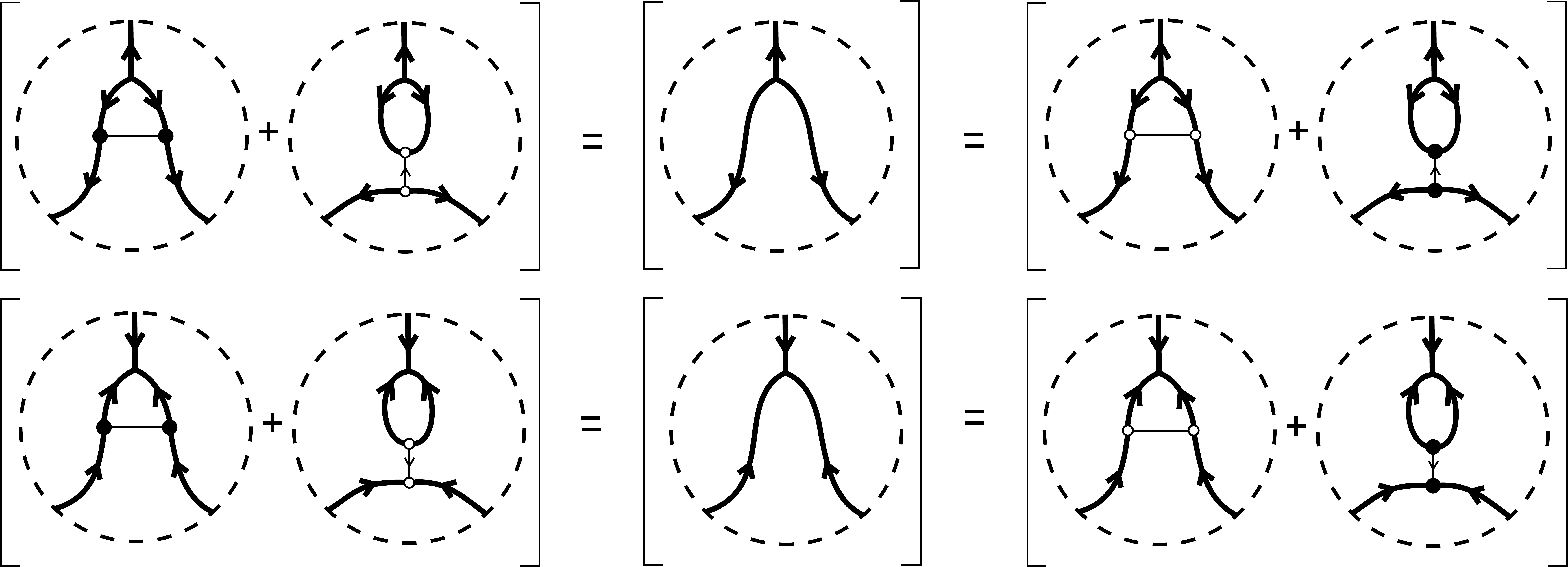}
\caption{Relations $R_{5.1}-R_{5.4}$ of the label bracket}
\label{fig4}
\end{center}
\end{figure}

\begin{Theorem} The (extended) label bracket is invariant under isotopy of knotted trivalent graphs.  
\end{Theorem}
\begin{proof} Invariance under moves $\Omega _{1}$, $\Omega _{2}$ and $\Omega _{3}$ is shown in \cite{AM}. Invariance under the move $\Omega _{4}$ follows from relations $R_{4.1}-R_{4.4}$, as shown in Fig. \ref{fig3}. Invariance under move $\Omega _{5}$ follows from relations $R_{5.1}-R_{5.4}$, as shown in Fig. \ref{fig5}.  
\end{proof}

\begin{figure}[h!]
\labellist
\normalsize \hair 2pt
\pinlabel $\mathcal{L}$ at 1725 0
\pinlabel $\mathcal{L}$ at 2685 0
\pinlabel $\mathcal{L}$ at 1725 460
\pinlabel $\mathcal{L}$ at 1200 460
\pinlabel $\mathcal{L}$ at 1725 940
\pinlabel $\mathcal{L}$ at 2685 940
\pinlabel $\mathcal{L}$ at 1725 1430
\pinlabel $\mathcal{L}$ at 1200 1430
\pinlabel $\mathcal{L}$ at 1725 1910
\pinlabel $\mathcal{L}$ at 2685 1910
\pinlabel $\mathcal{L}$ at 1725 2390
\pinlabel $\mathcal{L}$ at 1200 2390
\pinlabel $\mathcal{L}$ at 1725 2860
\pinlabel $\mathcal{L}$ at 2685 2860
\pinlabel $\mathcal{L}$ at 1725 3360
\pinlabel $\mathcal{L}$ at 1200 3360
\endlabellist
\begin{center}
\includegraphics[scale=0.18]{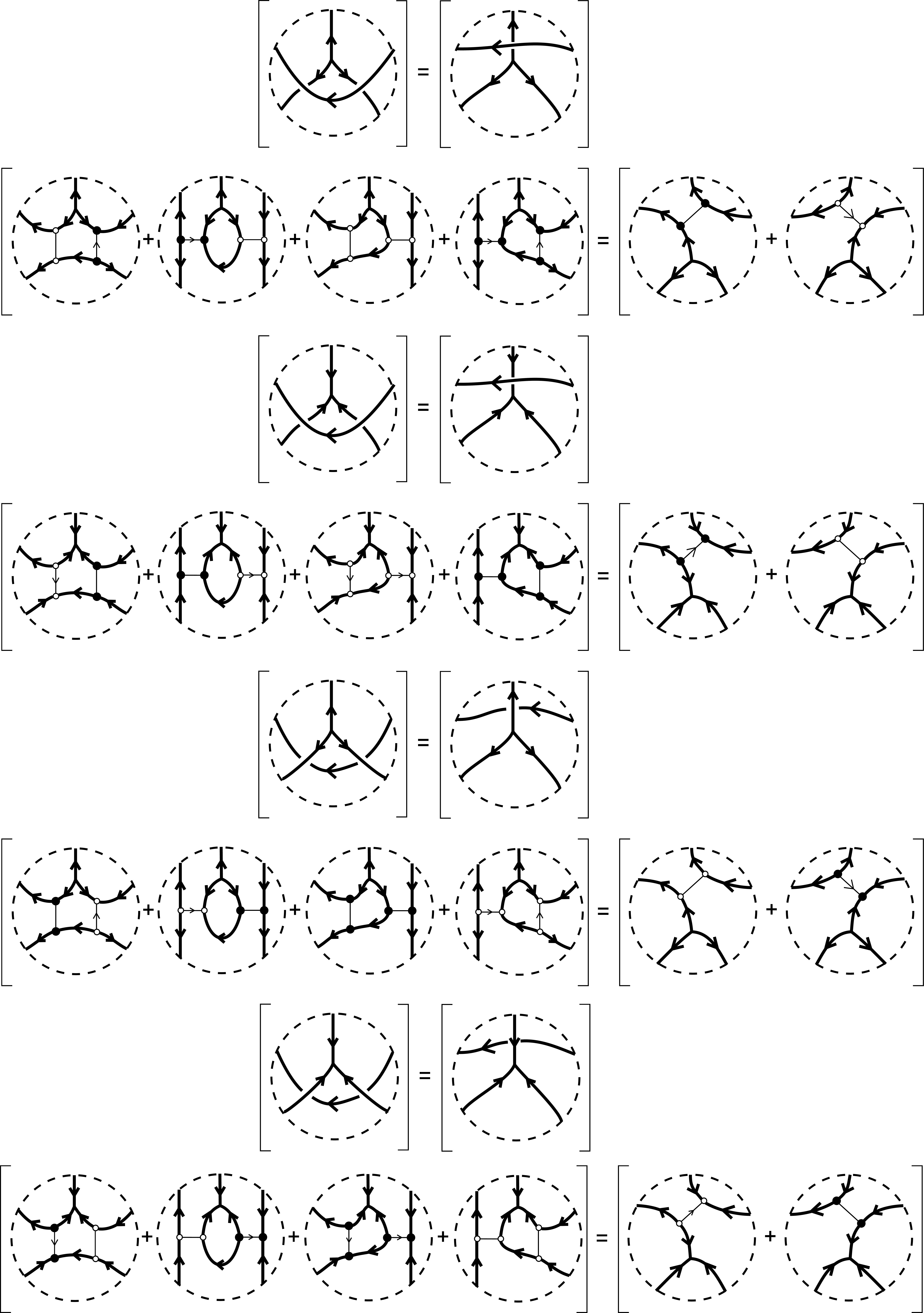}
\caption{Invariance of the label bracket under the move $\Omega _{4}$}
\label{fig3}
\end{center}
\end{figure}
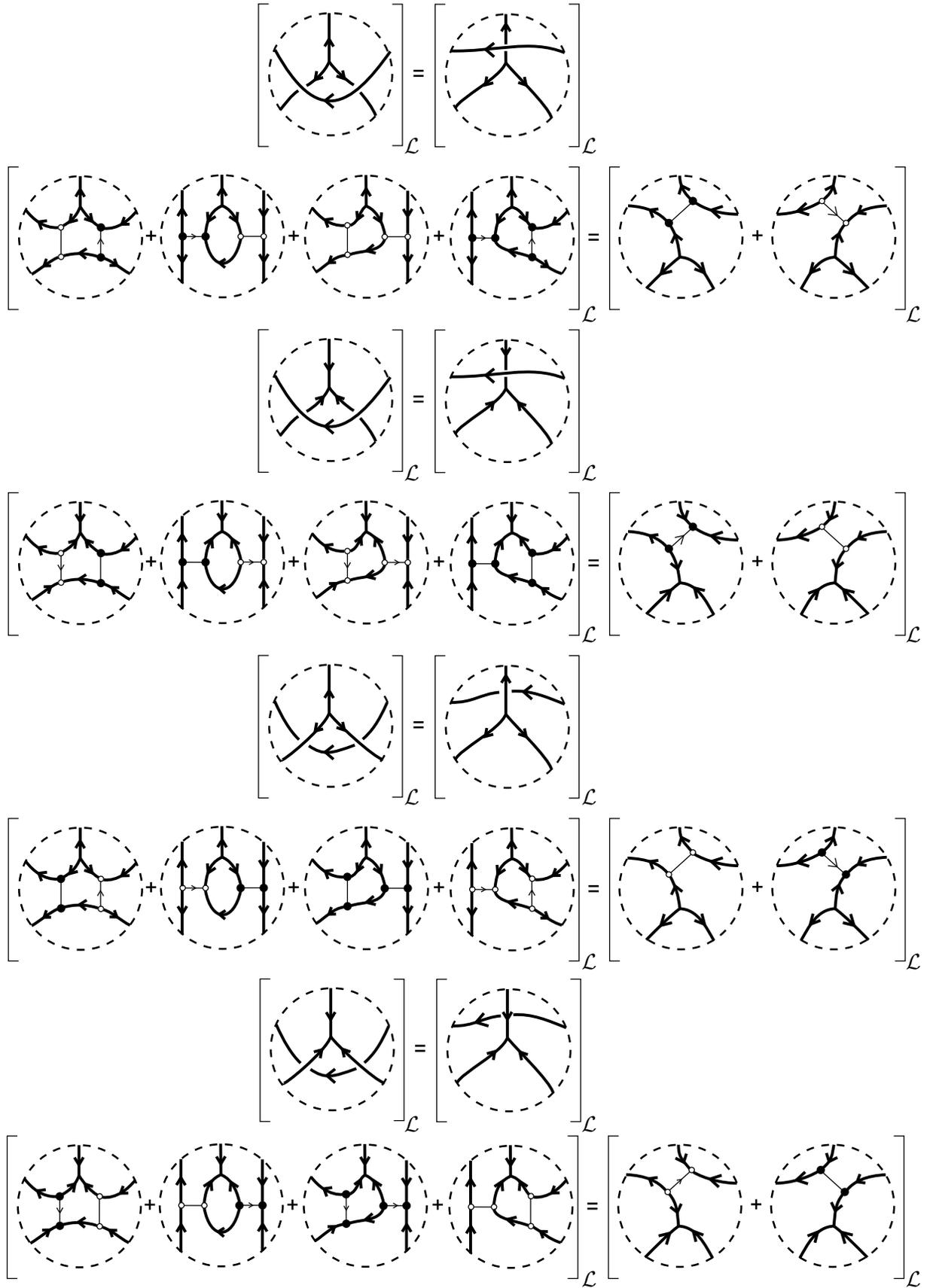

\begin{figure}[h!]
\labellist
\normalsize \hair 2pt
\pinlabel $\mathcal{L}$ at 890 0
\pinlabel $\mathcal{L}$ at 1490 0
\pinlabel $\mathcal{L}$ at 2510 0
\pinlabel $\mathcal{L}$ at 890 960
\pinlabel $\mathcal{L}$ at 1490 960
\pinlabel $\mathcal{L}$ at 2510 960
\pinlabel $\mathcal{L}$ at 890 480
\pinlabel $\mathcal{L}$ at 1490 480
\pinlabel $\mathcal{L}$ at 2090 480
\pinlabel $\mathcal{L}$ at 890 1430
\pinlabel $\mathcal{L}$ at 1490 1430
\pinlabel $\mathcal{L}$ at 2090 1430
\endlabellist
\begin{center}
\includegraphics[scale=0.18]{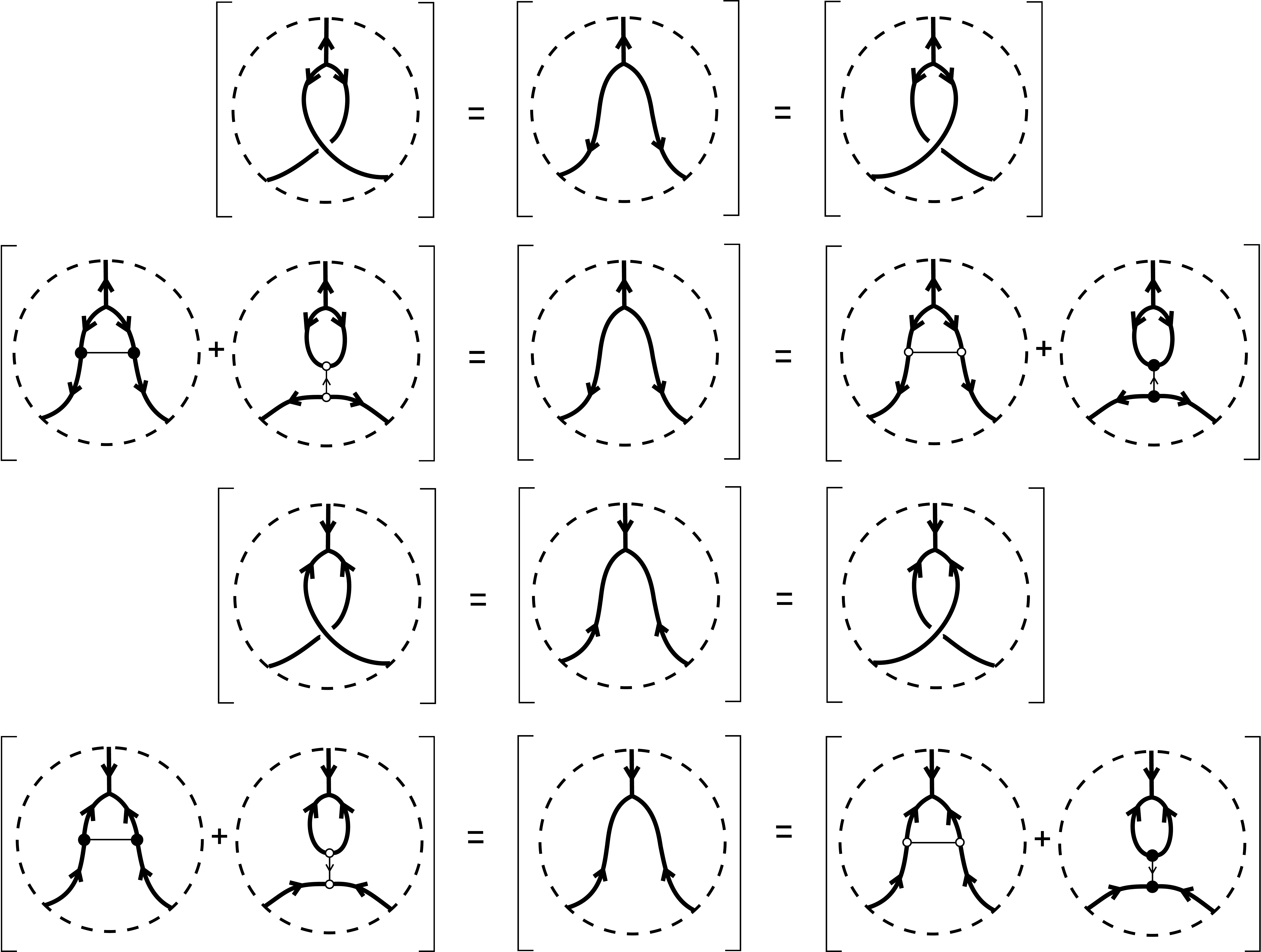}
\caption{Invariance of the label bracket under the move $\Omega _{5}$}
\label{fig5}
\end{center}
\end{figure}

\section*{Acknowledgments}

I would like to thank Vassily O. Manturov for providing me with valuable insight and pointing out the open question that triggered the writing of this paper. \\

The author was supported by the Slovenian Research Agency grant N1-0083.


\begin{thebibliography}{1}

\bibitem{AM} A. A. Akimova, V. O. Manturov, \textit{Labels instead of coefficients: a label bracket $[ \cdot ]_{\mathcal{L}(\cdot )}$ which dominates the Jones polynomial $\chi (\cdot )$, the Kuperberg bracket $|\cdot \rangle _{A_{2}}$ and the normalised arrow polynomial $\mathcal{W}[ \cdot ]$}, to appear in Journal of Knot Theory and its Ramifications, 2019.


\bibitem{KA} L. H. Kauffman, \textit{Invariants of graphs in three-space}, Transactions of the American Mathematical Society, \textbf{311} (2), 1989.



\end{thebibliography}
\end{document}